\documentclass[reqno]{amsart}
\usepackage{hyperref}
\newtheorem*{theoA}{Theorem A}
\newtheorem*{theoB}{Theorem B}
\newtheorem*{theoC}{Theorem C}

\newtheorem{theo}{Theorem}[section]
\newtheorem{lem}{Lemma}[section]

\newtheorem{ques}{Question}[section]
\newtheorem{exm}{Example}[section]

\newtheorem{rem}{Remark}[section]
\newcommand{\ol}{\overline}
\newcommand{\be}{\begin{equation}}
\newcommand{\ee}{\end{equation}}
\newcommand{\beas}{\begin{eqnarray*}}
	\newcommand{\eeas}{\end{eqnarray*}}
\newcommand{\bea}{\begin{eqnarray}}
\newcommand{\eea}{\end{eqnarray}}
\newcommand{\lra}{\longrightarrow}
\numberwithin{equation}{section}

\begin{document}
\title[\hfilneg \hfil  Uniqueness of entire functions]
{Uniqueness of an entire function sharing two values jointly with its differential polynomials}
\author[G. Haldar\hfil
\hfilneg]
{Goutam Haldar}


\address{Goutam Haldar  \newline
Department of Mathematics, Malda College, Rabindra Avenue, Malda, West Bengal 732101, India.}
\email{goutamiit1986@gmail.com, goutamiitm@gmail.com}


\subjclass[2010]{30D35}
\keywords{ Entire function, Derivative, Set sharing, Uniqueness}

\maketitle
\let\thefootnote\relax

\begin{abstract}
In this paper, we continue to investigate the uniqueness problem when an entire function $f$ and its linear differential polynomial $L(f)$ share two distinct complex values CMW (counting multiplicities in the weak sense) jointly. Also, We investigate the same problem when $f$ and its differential monomial $M(f)$ share two distinct complex values CMW. Our results generalize the recent result of Lahiri (Comput. Methods Funct. Theory, https://doi.org/10.1007/s40315-020-00355-4). 	
\end{abstract}
\section{Introduction, Definitions and Results}
A function analytic in the open complex plane $\mathbb{C}$ except possibly for poles is called meromorphic in $\mathbb{C}$. If no poles occur, then the function is called entire. For a non-constant meromorphic function $f$ defined in $\mathbb{C}$ and for $a\in\mathbb{C}\cup\{\infty\}$, we denote by $E(a,f)$ the set of $a$-points of $f$ counted multiplicities and $\ol E(a,f)$ the set of all $a$-points ignoring multiplicities. If for two non-constant meromorphic functions $f$ and $g$, $E(a,f)=E(a,g)$, we say that $f$ and $g$ share the value $a$ CM (counting multiplicities). If $\ol E(a,f)=\ol E(a,g)$, then we say that $f$ and $g$ are said to share the value $a$ IM (ignoring multiplicities). Throughout the paper, the standard notations of Nevanlinna's value distribution theory of meromorphic functions \cite{4, Yan & Yi & 2003} have been adopted. A meromorphic function $a(z)$ is said to be small with respect to $f$ provided that $T(r,a)=S(r,f)$, that is $T(r,a)=o(T(r,f))$ as $r\lra \infty$, outside of a possible exceptional set of finite linear measure. \vspace{1.5mm}

\par In $1976$, it was shown by Rubel and Yang \cite{Rubel & Yang & 1977} that if an entire function $f$ and its derivative $f^{\prime}$ share two values $a$, $b$ CM, then $f=f^{\prime}$. After that Gundersen \cite{Gundersen & 1980} improved the result by considering two IM shared Values. Yang \cite{Yang & 1990} also extended the result of Rubel and Yang \cite{Rubel & Yang & 1977} by replacing $f^{\prime}$ with the $k$-th derivative $f^{(k)}$. Since then the subject of sharing values between a meromorphic
function and its derivatives has become one of the most prominent branches of the uniqueness theory. Mues and Steinmetz \cite{Mues & Stein & 1979} showed that if a meromorphic function $f$ shares three finite values IM with $f^{\prime}$, then $f=f^{\prime}$. Frank and Schwick \cite{Frank & Schwick & 1992} improved this result by replacing $f^{\prime}$ with $f^{(k)}$, where $k$ is a positive integer. After that many mathematicians spent their times towards the improvements of this result (see \cite{Frank & Hua & 1999, Gu & 1994, Li & 2005, Mues & Rein & 1992}). In $2000$, Li
and Yang \cite{Li & Yang & 2000} improved the result of Yang \cite{Yang & 1990} in the following.

\begin{theoA}\cite{Li & Yang & 2000}
Let $f$ be a non-constant entire function, $k$ be a positive integer and $a$, $b$ be distinct finite numbers. If $f$ and $f^{(k)}$ share $a$ and $b$ IM, then $f=f^{(k)}$.	
\end{theoA}
We now recall the notion of set sharing as follows: Let $S$ be a subset of distinct elements of $\mathbb{C}\cup\{\infty\}$ and $E_f(S)=\bigcup_{a\in S}E(a,f)$ and $\ol E_f(S)=\bigcup_{a\in S}\ol E(a,f)$. We say that two meromorphic functions $f$ and $g$ share the set $S$ CM or IM if $E_f(S)=E_g(S)$ or $\ol E_f(S)=\ol E_g(S)$, respectively.\vspace{1mm}
\par Using the notion of set sharing instead of value sharing, Li and Yang \cite{Li & Yang 1999} proved the following theorem. 
\begin{theoB}\cite{Li & Yang 1999}
Let $f$ be a non-constant entire function and $a_1$, $a_2$ be two distinct finite complex numbers. If $f$ and $f^{(1)}$ share the set $\{a_1, a_2\}$ CM, then one and only one of the following holds:
\begin{enumerate}
\item [(i)] $f=f^{(1)}$
\item [(ii)] $f+f^{(1)}=a_1+a_2$
\item [(iii)] $f=c_1e^{cz}+c_2e^{-cz}$ with $a_1+a_2=0$, where $c$, $c_1$ and $c_2$ are non-zero constants satisfying $c^2\neq 1$ and $4c^2c_1c_2=a_1^2(c^2-1)$.
\end{enumerate}
\end{theoB}
In $2020$, Lahiri \cite{Lahiri & 2020} introduced a new type of set sharing notion called CMW (counting multiplicities in the weak sense)as follows: \vspace{1mm}
\par Let $f$ and $g$ be two non-constant meromorphic functions in $\mathbb{C}$ and $a\in\mathbb{C}\cup\{\infty\}$ and $B\subset \mathbb{C}\cup\{\infty\}$. We denote by $E_B(a;f,g)$ the set of those distinct $a$-points of $f$ which are the $b$-points of $g$ having the same multiplicity for some $b\in B$. For $A\subset\mathbb{C}\cup\{\infty\}$, we put $E_B(A;f,g)=\bigcup_{a\in A} E_B(a;f,g)$. Clearly $E_B(A;f,g)=E_B(A;g,f)$ for $A=B$. For $S\subset\mathbb{C}\cup\{\infty\}$ we define
\beas Y=\{\ol E(S,f)\cup\ol E(S,g)\}\setminus\ol E_S(S,;f,g).\eeas We say that $f$ and $g$ share the set $S$ with counting multiplicities in the weak
sense (CMW) if $N_Y(r,a;f)=S(r,f)$ and $N_Y(r,a;g)=S(r, g)$ for every $a\in S$, where $N_Y(r,a;f)$ denotes the counting function, counted with multiplicities of those $a$-points of $f$ which lie in the set $Y$.\vspace{1mm}
\par We note that $f$ and $g$ share the set $S$ with counting multiplicities if and only if $Y=\emptyset$.\vspace{1mm}
\par Lahiri \cite{Lahiri & 2020} greatly improved Theorem B by considering the higher order derivative $f^{(k)}$ and CMW in place of CM set sharing and proved the following theorem. 
\begin{theoC}\cite{Lahiri & 2020}
Let $f$ be a non-constant entire function and $k$ be a positive integer such that \bea\label{e1.1} \ol N\left(r,\frac{f^{(k)}}{f^{(1)}}\right)=S(r,f).\eea Suppose that $a_1$ and $a_2$ are two distinct finite complex numbers. If $f$ and $f^{(k)}$ share the set $\{a_1, a_2\}$ CMW, then only one of the following holds:
\begin{enumerate}
\item [(i)] $f=f^{(k)}$
\item [(ii)] $f+f^{(k)}=a_1+a_2$
\item [(iii)] $f=c_1e^{cz}+c_2e^{-cz}$ with $a_1+a_2=0$, where $c$, $c_1$ and $c_2$ are non-zero constants satisfying $c^{2k}\neq 1$ and $4c^{2k}c_1c_2=a_1^2(c^{2k}-1)$ and $k$ is an odd positive integer.
\end{enumerate}
\end{theoC}
\par For further investigation of the above theorem, we now define a linear differential polynomial $L(f)$ and a differential monomial $M(f)$ of an entire function $f$ as follows: \bea\label{e1.1a}L(f)=b_1f^{(1)}+b_2f^{(2)}+\cdots+b_kf^{(k)}=\sum_{j=1}^{k}b_jf^{(j)},\eea where $b_1,b_2,\ldots,b_k(\neq0)$ are complex constants, and \bea\label{e1.1b} M(f)=(f^{(1)})^{n_1}(f^{(2)})^{n_2}\cdots(f^{(k)})^{n_k},\eea where $k$ ia a positive integer and $n_1,n_2,\ldots, n_k$ are non-negative integers, not all of them are zero. We call $k$ and $\lambda=\sum_{j=1}^{k}n_j$, respectively the order and the degree of the monomial $M(f)$. 
\par From the above discussion it is natural to ask the following questions.
\begin{ques}
What can be said about the uniqueness when an entire function $f$ share two values jointly CMW with its linear differential polynomial $L(f)$?
\end{ques}
\begin{ques}
What can be said about the uniqueness of an entire function $f$ when $f$ share two values jointly CMW with its differential monomial $M(f)$?
\end{ques}
\par In the present paper, we prove the following results which will answer the above questions positively. We use a methodology which is similar to \cite{Lahiri & 2020} but with some modifications.
\section{Main results}
\begin{theo}\label{t1}
Let $f$ is a non-constant entire function and $L(f)$ be a linear differential polynomial defined as in $(\ref{e1.1a})$ such that  \bea\label{e2.1}\ol N\left(r,\frac{L(f)}{f^{(1)}}\right)=S(r,f).\eea Suppose that $a_1$ and $a_2$ are two distinct finite complex numbers. If $f$ and $L(f)$ share the set $\{a_1, a_2\}$ CMW, then only one of the following holds:
\begin{enumerate}
\item [(i)] $f=L(f)$
\item [(ii)] $f+L(f)=a_1+a_2$
\item [(iii)] $f=c_1e^{cz}+c_2e^{-cz}$ with $a_1+a_2=0$, where $c$, $c_1$ and $c_2$ are non-zero constants satisfying $(b_1c+b_3c^3+\cdots+b_kc^k)^2\neq 1$ and $4(b_1c+b_3c^3+\cdots+b_kc^k)^2c_1c_2=a_1^2((b_1c+b_3c^3+\cdots+b_kc^k)^2-1)$ and $k$ is an odd positive integer.
\end{enumerate}	
\end{theo}
\begin{theo}\label{t2}
Let $f$ is a non-constant entire function and $M(f)$ be a differential monomial defined as in $(\ref{e1.1b})$ such that  \bea\label{e2.1a}\ol N\left(r,\frac{M(f)}{(f^{\lambda})^{(1)}}\right)=S(r,f).\eea Suppose that $a_1$ and $a_2$ are two distinct finite complex numbers. If $f^{\lambda}$ and $M(f)$ share the set $\{a_1, a_2\}$ CMW, then only one of the following holds:
\begin{enumerate}
\item [(i)] $f^{\lambda}=M(f)$
\item [(ii)] $f^{\lambda}+M(f)=a_1+a_2$
\item [(iii)] $f^{\lambda}=c_1e^{cz}+c_2e^{-cz}$, $M(f)=\sqrt{A}(c_1e^{2cz}-c_2)/e^{cz}$ with $a_1+a_2=0$, where $A$, $c$, $c_1$ and $c_2$ are non-zero constants and $\lambda=\sum_{j=1}^{k}n_j$.
\end{enumerate}	
\end{theo}
We give the following examples in the support of the main theorems.
\begin{exm}
Let $f=e^{\omega z}+a_1+a_2$, where $\omega^{k}=-1$, $k$ is a positive integer and $a_1,\; a_2$ are any two finite distinct complex constants. and $L(f)=M(f)=f^{(k)}$. Then all the conditions of Theorems \ref{t1} and \ref{t2} are satisfied. Here conclusion $(ii)$ of Theorems \ref{t1} and \ref{t2} holds. 
\end{exm}
\begin{exm}
Let $f=e^{\lambda z}$, where $\lambda^{5}=1$ and $L(f)=M(f)=f^{(5)}$. Then all the conditions of the above two theorems are satisfied and conclusion $(i)$ of the above two theorems holds. 
\end{exm}
\begin{rem}
By taking $L(f)=f^{(k)}$ in Theorem \ref{t1}, we get Theorem C, which is a particular case of our result.
\end{rem}
\section{Key lemmas} In this section, we present some necessary lemmas which will be required to prove the main results.
\begin{lem}\label{lem3.1}
Let $f$ be a non-constant entire function and $a_1$, $a_2$ be two distinct finite complex numbers. If $f$ and $L(f)$ share the set $\{a_1, a_2\}$ CMW, then $S(r,L(f))=S(r ,f)$.
\end{lem}
\begin{proof}
Since $f$ is entire, we have \bea\label{e3.1} T(r,L(f))&=&m(r,L(f))\leq m\left(r,\frac{L(f)}{f}\right)+m(r,f)\nonumber\\&\leq& T(r,f)+S(r,f).\eea
Again since $f$ and $L(f)$ share the set $\{a_1,a_2\}$ CMW, we get by second fundamental theorem 
\bea\label{e3.2} T(r,f)&\leq& \ol N\left(r,\frac{1}{f-a_1}\right)+\ol N\left(r,\frac{1}{f-a_2}\right)+S(r,f)\nonumber\\&\leq&2T(r,L(f))+S(r,f).\eea From \ref{e3.1} and \ref{e3.2}, we conclude that $S(r,L(f))=S(r,f)$. This proves the lemma.
\end{proof}
\begin{lem}\label{lem3.6}
	Let $f$ be a non-constant entire function and $a_1$, $a_2$ be two distinct finite complex numbers. If $f^{\lambda}$ and $M(f)$, where $\lambda=\sum_{j=1}^{k}n_j$ and $M(f)$ is defined as in $(\ref{e1.1b})$ share the set $\{a_1, a_2\}$ CMW, then $S(r,M(f))=S(r,f)$.
\end{lem}
\begin{proof}
	The proof of the lemma can be carried out in the line of the proof of Lemma \ref{e3.1}. So, we omit the details.
\end{proof}
\begin{lem}\label{lem3.2}\cite{Mohonko & 1971, Yan & Yi & 2003}
Let $f$ be a non-constant meromorphic function and $R(f)=P(f)/Q(f)$, where $P(f)=\sum_{k=0}^{p}a_kf^{k}$ and $Q(f)=\sum_{j=0}^{q}b_jf^j$ are two mutually prime polynomials in $f$. If $T(r,a_k)=S(r,f)$ and $T(r,b_j)=S(r,f)$ for $k=0,1,2,\ldots,p$ and $j=0,1,2, \ldots, q$ and $a_p\not\equiv0$, $b_q\not\equiv0$, then
$T(r,R(f))=\text{max}\{p, q\}T(r,f)+S(r,f)$.
\end{lem}
\begin{lem}\label{lem3.3}\cite{Laine & 1993}
The coefficients $a_0(\not\equiv0), a_1,\ldots,a_{n-1}$ of the differential equation $f^{(n)}+a_{n-1}f^{(n-1)}+\cdots+a_1f^{(1)}+a_0f=0$ are polynomials if and only if all solutions of it are entire functions of finite order .
\end{lem}
\begin{lem}\label{lem3.4}
Let $f$ be a non-constant entire function and $a_1, a_2$ be two non-zero distinct finite numbers. If $f$ and $L(f)$ $(k\geq1)$ share the set $\{a_1,a_2\}$ CMW and $T(r,h)\neq S(r,f)$, where \bea\label{e3.3} h=\frac{(L(f)-a_1)(L(f)-a_2)}{(f-a_1)(f-a_2)},\eea then the following fold:
\begin{enumerate}
\item [(i)] $\Psi\not\equiv 0$ and $T(r,\Psi)=S(r,f)$, where \bea\label{e3.4} \Psi=\frac{(f^{(1)}h-L^{(1)}(f))(f^{(1)}h+L^{(1)}(f))}{(L(f)-a_1)(L(f)-a_2)}.\eea
\item [(ii)] $T(r,L(f))=N\left(r,\displaystyle\frac{1}{L(f)-a_j}\right)+S(r,f)$ for $j=1,2$.
\item [(iii)] $m\left(r,\displaystyle\frac{1}{f-c}\right)=S(r,f)$, where $c\neq a_1,a_2\in\mathbb{C}$.
\item [(iv)] \beas T(r,h)&=&m\left(r,\displaystyle\frac{1}{f-a_1}\right)+m\left(r,\displaystyle\frac{1}{f-a_2}\right)+S(r,f)\\&=& m\left(r,\displaystyle\frac{1}{f^{(1)}}\right)+S(r,f)\leq m\left(r,\displaystyle\frac{1}{L(f)}\right)+S(r,f).\eeas
\item [(v)] $2T(r,f)-2T(r,L(f))=m\left(r,\displaystyle\frac{1}{h}\right)+S(r,f)$.
\end{enumerate}\end{lem}
\begin{proof}
Since $f$ and $L(f)$ share the set $\{a_1,a_2\}$ CMW, $N(r,h)+N(r,1/h)=S(r,f)$. Now if $\Psi\equiv 0$, then $h=\pm L^{(1)}(f)/f^{(1)}$. This implies that $T(r,h)=S(r,f)$, which contradicts to our assumption. Therefore $\Psi\not\equiv0.$ \vspace{1mm}
\par Let $z_0$ be a zero of $(L(f)-a_1)(L(f)-a_2)$ and $(f-a_1)(f-a_2)$ of multiplicity $p(\geq2)$. Then $z_0$ is a zero of $(f^{(1)}h-L^{(1)})(f^{(1)}h+L^{(1)})$ with multiplicity $2(p-1)\geq p$. So, $z_0$ is not a pole of $\Psi$.\vspace{1mm}
\par From (\ref{e3.3}), we get 
\bea\label{e3.5} (L(f)-a_1)(L(f)-a_2)=h(f-a_1)(f-a_2).\eea
Differentiating (\ref{e3.5}), we obtain 
\bea\label{e3.6} L^{(1)}(f)(2L(f)-a_1-a_2)=h^{(1)}(f-a_1)(f-a_2)+hf^{(1)}(2f-a_1-a_2).\eea
Let $z_0$ be a simple zero of $(L(f)-a_1)(L(f)-a_2)$ and $(f-a_1)(f-a_2)$. Then 
\beas 2L(f)(z_1)-a_1-a_2=\pm(2f(z_1)-a_1-a_2).\eeas
So from (\ref{e3.6}), we get \beas (h(z_1)f^{(1)}(z_1)-(L(f)(z_1))^2)(h(z_1)f^{(1)}(z_1)+(L(f)(z_1))^2)=0.\eeas Hence from (\ref{e3.4}), we see that $z_1$ is not a pole of $\Psi$. Since $f$ and $L(f)$ share the set $\{a_1,a_2\}$ CMW, we obtain $N(r,\Psi)=S(r,f)$.\vspace{1mm}
\par By (\ref{e3.3}), we get \bea\label{e3.7} \frac{f^{(1)}h-L^{(1)}(f)}{L(f)-a_1}=\frac{f^{(1)}L(f)}{(f-a_1)(f-a_2)}-\frac{a_2f^{(1)}}{(f-a_1)(f-a_2)}-\frac{L^{(1)}(f)}{L(f)-a_1}.\eea
Since \beas \frac{a_2f^{(1)}}{(f-a_1)(f-a_2)}=\frac{1}{a_1-a_2}\left(\frac{f^{(1)}}{f-a_1}-\frac{f^{(1)}}{f-a_2}\right),\eeas we get from (\ref{e3.7}) that \beas m\left(r,\frac{f^{(1)}h-L^{(1)}(f)}{L(f)-a_1}\right)=S(r,f).\eeas Similarly, \beas m\left(r,\frac{f^{(1)}h+L^{(1)}(f)}{L(f)-a_2}\right)=S(r,f).\eeas Therefore, from (\ref{e3.4}) we obtain $m(r,\Psi)=S(r,f)$ and hence $T(r,\Psi)=S(r,f)$, which is $(i)$.\vspace{1mm}
\par Now in view of (\ref{e3.3}), we get from (\ref{e3.4}) that \beas \frac{1}{f^{(1)}h-L^{(1)}(f)}=\frac{1}{\Psi}\left(\frac{f^{(1)}}{(f-a_1)(f-a_2)}+\frac{L^{(1)}(f)}{(L(f)-a_1)(L(f)-a_2)}\right).\eeas Therefore, \beas m\left(r,\frac{1}{f^{(1)}h-L^{(1)}(f)}\right)=S(r,f).\eeas Similarly, we get \beas m\left(r,\frac{1}{f^{(1)}h+L^{(1)}(f)}\right)=S(r,f).\eeas So we obtain \beas m\left(r,\frac{1}{L(f)-a_1}\right)&\leq& m\left(r,\frac{f^{(1)}h-L^{(1)}(f)}{L(f)-a_1}\right)+m\left(r,\frac{1}{f^{(1)}h-L^{(1)}(f)}\right)\\&=&S(r,f)\eeas and $m\left(r,\displaystyle\frac{1}{L(f)-a_2}\right)=S(r,f)$. Therefore, \beas T(r, L(f))=N\left(r, \frac{1}{L(f)-a_j}\right)+S(r,f),\eeas for $j=1,2$, which is $(ii)$.\vspace{1mm}
\par for $c\neq a_1, a_2$, we get from (\ref{e3.7}) \beas \frac{f^{(1)}h-L^{(1)}(f)}{(L(f)-a_1)(f-c)}&=&\frac{f^{(1)}L(f)}{(f-c)(f-a_1)(f-a_2)}-\frac{a_2f^{(1)}}{(f-c)(f-a_1)(f-a_2)}\\&&-\frac{L^{(1)}(f)}{(L(f)-a_1)(f-c)}.\eeas We note that \beas \frac{a_2f^{(1)}}{(f-c)(f-a_1)(f-a_2)}=\alpha\frac{f^{(1)}}{f-c}+\beta\frac{f^{(1)}}{f-a_1}+\gamma\frac{f^{(1)}}{f-a_2},\eeas where $\alpha=\displaystyle\frac{a_2}{(a_1-c)(a_2-c)}$, $\beta=\displaystyle\frac{a_2}{(c-a_1)(a_2-a_1)}$ and $\gamma=\displaystyle\frac{a_2}{(c-a_2)(a_1-a_2)}$.
Therefore, we get \beas m\left(r,\frac{f^{(1)}h-L^{(1)}(f)}{(f-c)(L(f)-a_1)}\right)=S(r,f).\eeas Since by (\ref{e3.4}), \beas \frac{1}{f-c}=\frac{1}{\Psi}\frac{f^{(1)}h-L^{(1)}(f)}{(f-c)(L(f)-a_1)}\frac{f^{(1)}h+L^{(1)}(f)}{(L(f)-a_2)},\eeas we get \beas m\left(r,\frac{1}{f-c}\right)=S(r,f),\eeas which is $(iii)$. Since \beas h=\frac{(L(f))^2-(a_1+a_2)L(f)+a_1a_2}{(f-a_1)(f-a_2)},\; \text{we have}\eeas 
\bea\label{e3.8} T(r,h)&=&m(r,h)+S(r,f)\leq m\left(r,\frac{1}{(f-a_1)(f-a_2)}\right)+S(r,f)\nonumber\\&\leq& m\left(r,\frac{1}{f^{(1)}}\right)+m\left(r,\frac{f^{(1)}}{(f-a_1)(f-a_2)}\right)+S(r,f)\nonumber\\&\leq& m\left(r,\frac{1}{f^{(1)}}\right)+S(r,f).\eea
Since \beas \frac{\Psi}{f^{(1)}}&=&\frac{f^{(1)}}{(f-a_1)(f-a_2)}\frac{(L(f))^2-(a_1+a_2)L(f)+a_1a_2}{(f-a_1)(f-a_2)}\\&&-\frac{L^{(1)}(f)}{f^{(1)}}\frac{L^{(1)}(f)}{(L(f)-a_1)(L(f)-a_2)},\eeas we get by (i) that \bea\label{e3.9} m\left(r,\frac{1}{f^{(1)}}\right)&\leq& m\left(r,\frac{\Psi}{f^{(1)}}\right)+S(r,f)\nonumber\\&\leq& m\left(r,\frac{1}{(f-a_1)(f-a_2)}\right)+S(r,f).\eea
Since $\displaystyle\frac{1}{(f-a_1)(f-a_2)}=\displaystyle\frac{h}{(L(f)-a_1)(L(f)-a_2)}$, we have by (ii) that \bea\label{e3.10} m\left(r,\frac{1}{(f-a_1)(f-a_2)}\right)\leq T(r,h)+S(r,f).\eea From (\ref{e3.8}), (\ref{e3.9}) and (\ref{e3.10}), we have \beas T(r,h)=m\left(r,\frac{1}{f-a_1}\right)+m\left(r,\frac{1}{f-a_2}\right)+S(r,f)\leq m\left(r,\frac{1}{L(f)}\right)+S(r,f),\eeas which is (iv).\vspace{1mm}
\par Keeping in view of (\ref{e3.3}), we get from (ii) and (iv) that \beas 2T(r,L(f))&=&N\left(r,\frac{1}{L(f)-a_1}\right)+N\left(r,\frac{1}{L(f)-a_2}\right)+S(r,f)\\&=& N\left(r,\frac{1}{(L(f)-a_1)(L(f)-a_2)}\right)+S(r,f)\\&=&N\left(r,\frac{1}{h(f-a_1)(f-a_2)}\right)+S(r,f)\\&=&2T(r,f)-m\left(r,\frac{1}{f-a_1}\right)-m\left(r,\frac{1}{f-a_2}\right)+N\left(r,\frac{1}{h}\right)+S(r,f)\\&=&2T(r,f)-T(r,h)+N\left(r,\frac{1}{h}\right)+S(r,f).\eeas So, $2T(r,f)-2T(r,L(f))=m\left(r,\displaystyle\frac{1}{h}\right)+S(r,f)$, which is (v). This completes the proof of the lemma.
\end{proof}
\begin{lem}\label{lem3.5}
Let $f$ be a non-constant entire function and $a_1$, $a_2$ be two distinct finite complex numbers. If $f$ and $L(f)$ share the set $\{a_1,a_2\}$ CMW, then $T(r,h)=S(r,f)$, where $h$ is defined in Lemma \ref{lem3.4}.
\end{lem}
\begin{proof}
Since $f$ and $L(f)$ share the set $\{a_1,a_2\}$ CMW, we must have $N(r,h)=S(r,f)$ and $N\left(r,1/h\right)=S(r,f)$. Assume on the contrary that $T(r,h)\neq S(r,f)$. By Lemma \ref{lem3.4}, we know that $\Psi\not\equiv 0$ and $T(r,\Psi)=S(r,f)$.\vspace{1mm}
	
\par Differentiating (\ref{e3.3}), we get 
\bea\label{e3.11} 2L(f)L^{(1)}(f)-(a_1+a_2)L^{(1)}(f)&=&(2ff^{(1)}-(a_1+a_2)f^{(1)})h\nonumber\\&&+h^{(1)}(f-a_1)(f-a_2).\eea
From (\ref{e3.3}) and (\ref{e3.11}), we obtain \beas \frac{(2L(f)-(a_1+a_2))L^{(1)}(f)}{(L(f)-a_1)(L(f)-a_2)}=\frac{(2f-(a_1+a_2))f^{(1)}}{(f-a_1)(f-a_2)}+\frac{h^{(1)}}{h}.\eeas Squaring the above equation, we get \beas \frac{((2L(f)-(a_1+a_2))^2(L^{(1)}(f))^2}{(L(f)-a_1)^2(L(f)-a_2)^2}&=&\frac{((2f-(a_1+a_2)))^2(f^{(1)})^2}{(f-a_1)^2(f-a_2)^2}+\beta^2\\&&+\frac{2\beta(2f-(a_1+a_2))f^{(1)}}{(f-a_1)(f-a_2)},\eeas where $\beta=h^{(1)}/h$.\vspace{1mm}
\par Eliminating $L^{(1)}(f)$ from (\ref{e3.3}), (\ref{e3.4}) and the above equation, we get
\bea\label{e3.12} \frac{(2L(f)-(a_1+a_2))^2\Psi}{(L(f)-a_1)(L(f)-a_2)}&=&\frac{4(L(f)+f-(a_1+a_2))(L(f)-f)(f^{(1)})^2}{(f-a_1)^2(f-a_2)^2}-\beta^2\nonumber\\&&-\frac{2\beta(2f-(a_1+a_2))f^{(1)}}{(f-a_1)(f-a_2)}.\eea 
Let $z_0$ be a zero of $(f-a_1)(f-a_2)$ which is also a zero of $(L(f)-a_1)(L(f)-a_2)$. Since $f$ and $L(f)$ share the set $\{a_1,a_2\}$ CMW and $T(r,\beta)=S(,f)$, almost all the poles of right hand side of (\ref{e3.12}) are simple, and hence it follows from the same equation that ``almost all" the zeros of $(L(f)-a_1)(L(f)-a_2)$ are simple as long as they are not the zeros of $\Psi$. Thus \bea\label{e3.13} N\left(r,\displaystyle\frac{1}{L(f)-a_j}\right)=\ol N\left(r,\displaystyle\frac{1}{L(f)-a_j}\right)+S(r,f),\; \text{j=1, 2}.\eea 
Differentiating (\ref{e3.4}), we get 
\beas 2h^2f^{(1)}(f^{(2)}+\beta f^{(1)})-2L^{(1)}(f)L^{(2)}(f)&=&\Psi^{(1)}(L(f)-a_1)(L(f)-a_2)\\&&+\Psi(2L(f)-(a_1+a_2))L^{(1)}(f).\eeas
Now eliminating $h$ from the above equation by using (\ref{e3.4}), we get 
\bea\label{e3.14} &&\left[2\Psi(f^{(2)}+\beta f^{(1)})-f^{(1)}\Psi^{(1)}\right](L(f)-a_1)(L(f)-a_2)=\nonumber\\&&L^{(1)}\left[2f^{(1)}\Psi L-(a_1+a_2)f^{(1)}\Psi-2(\beta f^{(1)}+f^{(2)})L^{(1)}+2f^{(1)}L^{(2)}\right].\eea
From the above equation, we see that any simple zeros of $(L(f)-a_1)(L(f)-a_2)$ must be the zeros of $2f^{(1)}\Psi L(f)-(a_1+a_2)f^{(1)}\Psi-2(\beta f^{(1)}+f^{(2)})L^{(1)}(f)+2f^{(1)}L^{(2)}(f)$.\vspace{1mm}
\par Let \bea\label{e3.15} \Psi_1=\frac{2f^{(1)}\Psi L(f)-(a_1+a_2)f^{(1)}\Psi-2(\beta f^{(1)}+f^{(2)})L^{(1)}(f)+2f^{(1)}L^{(2)}(f)}{(f-a_1)(f-a_2)}.\eea 
Since $f$ and $L(f)$ share the set $\{a_1,a_2\}$ CMW and ``almost all" the zeros of $(L(f)-a_1)(L(f)-a_2)$ are simple, we must have $N(r,\Psi_1)=S(r,f)$.\vspace{1mm}
\par On the hand, by the lemma of logarithmic derivative, it can be easily seen that $m(r,\Psi_1)=S(r,f)$. Hence, $T(r,\Psi_1)=S(r,f)$.\vspace{1mm}
\par We now consider the following two cases:\vspace{1mm}
\par \textbf{Case 1:} $\Psi_1\not\equiv 0$. Then it follows from (\ref{e3.15}) that \beas 2T(r,f)&=&T(r,(f-a_1)(f-a_2))+S(r,f)\\&=&m(r,(f-a_1)(f-a_2))+S(r,f)\\&\leq& m(r,(f-a_1)(f-a_2)\Psi_1)+m\left(r,\frac{1}{\Psi_1}\right)+S(r,f)\\&\leq& m(r,f^{(1)})+m(r,L(f))+T(r,\Psi_1)+S(r,f)\\&\leq& T(r,f)+T(r,L(f))+S(r,f).\eeas Therefore, $T(r,f)\leq T(r,L(f))+S(r,f)$.\vspace{1mm}
\par Since $L(f)$ is a linear differential polynomial in $f$, we get \beas T(r,L(f))\leq T(r,f)+S(r,f).\eeas Combining the above two we have $T(r,f)=T(r,L(f))+S(r,f)$.\vspace{1mm}
\par By Lemma \ref{lem3.4} (ii) and (\ref{e3.13}), we get \beas 2T(r,f)=\ol N\left(r,\frac{1}{f-a_1}\right)+\ol N\left(r,\frac{1}{f-a_2}\right)+S(r,f),\eeas which implies that \beas m\left(r,\frac{1}{f-a_1}\right)+m\left(r,\frac{1}{f-a_2}\right)=S(r,f).\eeas Thus from (\ref{e3.3}), the lemma of logarithmic derivative and the above observation, we get \beas T(r,h)&=&N(r,h)+m(r,h)\\&\leq& m\left(r,\frac{(L(f))^2}{(f-a_1)(f-a_2)}\right)+m\left(r,\frac{a_1a_2}{(f-a_1)(f-a_2)}\right)\\&&+2m\left(r,\frac{L(f)}{(f-a_1)(f-a_2)}\right)+S(r,f)=S(r,f).\eeas i.e., $T(r,h)=S(r,f)$, which contradicts to our assumption.\vspace{1mm}
\par \textbf{Case 2:} $\Psi_1\equiv0$. Then from (\ref{e3.14}) and (\ref{e3.15}), we obtain \beas \frac{\Psi^{(1)}}{\Psi}=2\left(\frac{h^{(1)}}{h}+\frac{f^{(2)}}{f^{(1)}}\right).\eeas Integrating above, we get \bea\label{e3.16} (hf^{(1)})^2=c\Psi,\eea where $c$ is a non-zero constant.\vspace{1mm}
\par It follows from (\ref{e3.4}) and (\ref{e3.16}) that \beas (L^{(1)}(f))^2&=&-((L(f))^2-(a_1+a_2)L(f)+(a_1a_2-c))\Psi\\&=&-(L(f)-d_1)(L(f)-d_2)\Psi,\eeas where $d_1$ and $d_2$ are two complex constants. If $d_1\neq d_2$, then by the lemma of logarithmic derivative, we get \beas m\left(r,\frac{1}{L^{(1)}(f)}\right)=m\left(r,\frac{-L^{(1)}(f)}{(L(f)-d_1)(L(f)-d_2)}\right)=S(r,f).\eeas
Therefore, \beas m\left(r,\frac{1}{(f-a_1)(f-a_2)}\right)\leq m\left(r,\frac{L^{(1)}(f)}{(f-a_1)(f-a_2)}\right)+m\left(r,\frac{1}{L^{(1)}(f)}\right)=S(r,f).\eeas Hence keeping in view of the above, we get from (\ref{e3.3}) and the lemma of logarithmic derivative 
\beas T(r,h)&=&N(r,h)+m(r,h)\\&\leq& m\left(r,\frac{(L(f))^2}{(f-a_1)(f-a_2)}\right)+m\left(r,\frac{L(f)}{(f-a_1)(f-a_2)}\right)\\&&+m\left(r,\frac{1}{(f-a_1)(f-a_2)}\right)+S(r,f)\\&=& S(r,f),\eeas which contradicts to our assumption.\vspace{1mm}
\par Therefore, $d_1=d_2=(a_1+a_2)/2=d$, say.
Hence, \bea\label{e3.17} (L^{(1)}(f))^2=-\Psi(L(f)-d)^2.\eea From (\ref{e3.3}), (\ref{e3.16}) and (\ref{e3.17}), we get \bea\label{e3.18} (L(f)-d)(L(f)-a_1)(L(f)-a_2)=c_2(f-a_1)(f-a_2)\Psi_2,\eea where $c_2$ is a non-zero constant satisfying $c_2^2=-c$ and $\Psi_2=L^{(1)}(f)/f^{(1)}$.\vspace{1mm}
\par From (\ref{e3.16}), it can be easily seen that $N(r,1/f^{(1)})=S(r,f)$. Therefore, $N(r,\Psi_2)=S(r,f)$. On the other hand, by the lemma of logarithmic derivative, we have $m(r,\Psi_2)=S(r,f)$, and hence $T(r,\Psi_2)=S(r,f)$.\vspace{1mm}
\par Since $\Psi_2\not\equiv0$, it follow from (\ref{e3.18}) that \bea\label{e3.19} 3T(r,L(f))=2T(r,f)+S(r,f).\eea Let \bea\label{e3.20}\Psi_3=\frac{L^{(1)}(f)}{L(f)-d}.\eea Then from (\ref{e3.17}), we get $\Psi_3^2=-\Psi$. Hence, $T(r,\Psi_3)=S(r,f)$ and $\Psi_3\equiv0.$\vspace{1mm}
\par Now from (\ref{e3.3}) and (\ref{e3.11}), we get 
\bea\label{e3.21} (f-d)hf^{(1)}=(L(f)-d)^2\Psi_3-\frac{1}{2}\beta(L(f)-a_1)(L(f)-a_2).\eea By (\ref{e3.16}), we get \beas T(r,hf^{(1)})=S(r,f).\eeas Therefore, by (\ref{e3.21}), we obtain 
\bea\label{e3.22} T(r,f)=T(r,L(f)),\; \text{or}\; T(r,f)=2T(r,L(f))+S(r,f)\eea according as when $\Psi_3=\beta/2$ or not.\vspace{1mm}
\par Combining (\ref{e3.19}) and (\ref{e3.22}), we get $T(r,f)=S(r,f)$, which is a contradiction.\vspace{1mm}
\par Hence $T(r,h)=S(r,f)$. This completes the proof of the lemma.
\end{proof}
\begin{lem}\label{lem3.7}
Let $f$ be a non-constant entire function and $a_1, a_2$ be two non-zero distinct finite numbers. If $f^{\lambda}$ and $M(f)$ $(k\geq1)$ share the set $\{a_1,a_2\}$ CMW and $T(r,h_1)\neq S(r,f)$, where \bea\label{e3.23} h_1=\frac{(M(f)-a_1)(M(f)-a_2)}{(f^{\lambda}-a_1)(f^{\lambda}-a_2)},\eea then the following fold:
\begin{enumerate}
\item [(i)] $\Phi\not\equiv 0$ and $T(r,\Phi)=S(r,f)$, where \bea\label{e3.24} \Phi=\frac{((f^{\lambda})^{(1)}h_1-(M(f))^{(1)})((f^{\lambda})^{(1)}h_1+(M(f))^{(1)})}{(M(f)-a_1)(M(f)-a_2)}.\eea
\item [(ii)] $T(r,M(f))=N\left(r,\displaystyle\frac{1}{M(f)-a_j}\right)+S(r,f)$ for $j=1,2$.
\item [(iii)] $m\left(r,\displaystyle\frac{1}{f^{\lambda}-c}\right)=S(r,f)$, where $c\neq a_1,a_2\in\mathbb{C}$.
\item [(iv)] \beas T(r,h_1)&=&m\left(r,\displaystyle\frac{1}{f^{\lambda}-a_1}\right)+m\left(r,\displaystyle\frac{1}{f^{\lambda}-a_2}\right)+S(r,f)\\&=& m\left(r,\displaystyle\frac{1}{(f^{\lambda})^{(1)}}\right)+S(r,f)\leq m\left(r,\displaystyle\frac{1}{M(f)}\right)+S(r,f).\eeas
\item [(v)] $2\lambda T(r,f)-2T(r,M(f))=m\left(r,\displaystyle\frac{1}{h_1}\right)+S(r,f)$.
\end{enumerate}\end{lem}
\begin{proof}
The proof of this lemma can be carried out in a similar manner as done in the proof of Lemma \ref{lem3.4}. So, we omit the details.
\end{proof}
\begin{lem}\label{lem3.8}
Let $f$ be a non-constant entire function and $a_1$, $a_2$ be two distinct finite complex numbers. If $f^{\lambda}$ and $M(f)$ share the set $\{a_1,a_2\}$ CMW, then $T(r,h_1)=S(r,f)$, where $h_1$ is defined in Lemma \ref{lem3.7}.
\end{lem}
\begin{proof}
The proof of this lemma is essentially can be done in a similar manner as Lemma \ref{lem3.5}. So, we omit the details.	
\end{proof}

\section{Proof of the main results}
\begin{proof}[\bf Proof of Theorem \ref{t1}]
Let $2\eta$ be the principal branch of $\log h$, where $h$ is defined as in Lemma \ref{lem3.4}. Then by Lemma \ref{lem3.5}, we obtain \beas T(r,e^{\eta})=\frac{1}{2}T(r,h)+S(r,f)=S(r,f).\eeas Also (\ref{e3.3}) can be written as \bea\label{e4.1} (L(f)-a_1)(L(f)-a_2)=e^{2\eta}(f-a_1)(f-a_2).\eea And so \bea\label{e4.2} GH=\left(\frac{a_1-a_2}{2}\right)^2(e^{2\eta}-1),\eea where \beas G=e^{\eta}f-\frac{a_1+a_2}{2}e^{\eta}+L(f)-\frac{a_1+a_2}{2}\eeas and \beas H=e^{\eta}f-\frac{a_1+a_2}{2}e^{\eta}-L(f)+\frac{a_1+a_2}{2}.\eeas
If $e^{2\eta}\equiv1$, then from (\ref{e4.1}), we get \beas (f-L(f))(f+L(f)-a_1-a_2)=0,\eeas which implies that either $f=L(f)$, or $f+L(f)=a_1+a_2$.\vspace{1mm}
\par Now suppose that $e^{2\eta}\not\equiv1$. Since $f$ is entire we get $N(r,G)+N(r,H)=S(r, f)$, and so, from (\ref{e4.2}), we get $N(r,1/H)+N(r,1/G)=S(r,f)$.
Therefore, \bea\label{e4.3} T\left(r,\frac{G^{(j)}}{G}\right)+T\left(r,\frac{H^{(j)}}{H}\right)=S(r,f),\eea where $j=1,2,\ldots,k$.\vspace{1mm}
\par Suppose $f^{(1)}=b L(f)$. Then using the condition (\ref{e1.1}), the lemma of logarithmic derivative, and the first fundamental theorem of Nevalinna, it is easily seen that $T(r,b)=S(r,f)$.
\par From the definition of $G$ and $H$ it follows that \bea\label{e4.4} G+H=e^{\eta}(2f-a_1-a_2)\eea and \bea\label{e4.5} G-H=2L(f)-a_1-a_2=2\lambda f^{(1)}-a_1-a_2,\eea where $b\lambda=1$ and $T(r,\lambda)=S(r,f)$ as $T(r,b)=S(r,f).$\vspace{1mm}
\par Eliminating $f$ and $f^{(1)}$, from $(\ref{e4.4})$ and $(\ref{e4.5})$, we get \bea\label{e4.6} \left(e^{\eta}+\lambda\eta^{(1)}-\lambda\frac{G^{(1)}}{G}\right)G+\left(\lambda\eta^{(1)}-e^{\eta}-\lambda\frac{H^{(1)}}{H}\right)H+b(a_1+a_2)=0.\eea
Now eliminating $H$ from $(\ref{e4.2})$ and $(\ref{e4.6})$, we obtain
\bea\label{e4.7} \Phi_1 G^{2}+\Phi_2 G+\Phi_3=0,\eea where \bea\label{e4.8} \Phi_1=e^{\eta}+\lambda\eta^{(1)}-\lambda\frac{G^{(1)}}{G},\eea \bea\label{e4.9}\Phi_2= \lambda\eta^{(1)}-e^{\eta}-\lambda\frac{H^{(1)}}{H}\left(\frac{a_1-a_2}{2}\right)^2(e^{2\eta}-1),\eea
\bea\label{e4.10} \Phi_3=\lambda(a_1+a_2).\eea
If $\Phi_1\not\equiv0$ or $\Phi_2\not\equiv0$, then by Lemma \ref{e3.2}, we see from $(\ref{e4.7})$ that $T(r,G)=S(r,f)$, and therefore from $(\ref{e4.4})$, we get $T(r,f)=S(r,f)$, which is a contradiction.
Therefore, $\Phi_1=\Phi_2=0$. Then from $(\ref{e4.7})$, we get $\Phi_3=0$. This implies that 
\bea\label{e4.11} e^{\eta}+\lambda\eta^{(1)}-\lambda\frac{G^{(1)}}{G}=0,\eea	
\bea\label{e4.12} \lambda\eta^{(1)}-e^{\eta}-\lambda\frac{H^{(1)}}{H}=0,\eea	
\bea\label{e4.13} a_1+a_2=0.\eea	
Adding $(\ref{e4.11})$ and $(\ref{e4.12})$, we get \beas \frac{G^{(1)}}{G}+\frac{H^{(1)}}{H}=2\eta^{(1)}, \eeas and so by integration , we have \bea\label{e4.14}GH=c_0e^{2\eta},\eea where $c_0$ is a non-zero constant.\vspace{1mm} 
\par Now from $(\ref{e4.2})$, $(\ref{e4.13})$ and $(\ref{e4.14})$, we get $e^{2\eta}=A$, where $A$ is a constant.\vspace{1mm}
\par From $(\ref{e4.4})$, $(\ref{e4.5})$ and $(\ref{e4.13})$, we get \bea\label{e4.15} \left(\sqrt{A}-\sum_{j=1}^{k}b_j\frac{G^{(j)}}{G}\right)G^{2}=\left(\sqrt{A}+\sum_{j=1}^{k}b_j\frac{H^{(j)}}{H}\right)B,\eea where $B=(a_1-a_2)^2/4(A-1)$, constant.\vspace{1mm}
\par If $\sqrt{A}-\sum_{j=1}^{k}b_jG^{(j)}/G\not\equiv0,$ then from $(\ref{e4.3})$ and $(\ref{e4.15})$, we get $T(r,G)=S(r,f)$ and so from $(\ref{e4.14})$, we get $T(r,F)=S(r,f)$. Therefore, from $(\ref{e4.4})$, we get $T(r,f)=S(r,f)$, which is a contradiction. hence we have $\sum_{j=1}^{k}b_jG^{(j)}-\sqrt{A}G=0$ and $\sum_{j=1}^{k}b_jH^{(j)}+\sqrt{A}H=0$. This implies by Lemma \ref{lem3.3} that $G$ and $H$ are of finite order. Also from $(\ref{e4.14})$, we see that $G$ and $H$ do not assume the value $0$.\vspace{1mm} \par Therefore, let us assume that $G=e^{P}$ and $H=e^{Q}$, where $P,\; Q$ are polynomials of degree $p$ and $q$, respectively. Differentiating $j$ times, we obtain $G^{(j)}=P_je^{P}$ and $H^{(j)}=Q_je^{Q}$, where $P_j$ and $Q_j$ are polynomials of degree $(p-1)j$ and $(q-1)j$, respectively. Since $\sum_{j=1}^{k}b_jG^{(j)}=\sqrt{A}G$ and $\sum_{j=1}^{k}b_jH^{(j)}=\sqrt{A}H$, we have $p=q=1$. Hence in view of $(\ref{e4.14})$, we may write $G=2d_1e^{cz}$ and $H=2d_2e^{-cz}$, where $c,\;d_1,\;d_2$ are non-zero constants.\vspace{1mm}
\par Now from $(\ref{e4.4})$ and  $(\ref{e4.13})$, we get \bea\label{e4.16} f=c_1e^{cz}+c_2e^{-cz},\eea where $c_1=d_1/\sqrt{A}$ and $c_2=d_2/\sqrt{A}$.\vspace{1mm}
\par Differentiating $(\ref{e4.16})$, we have \beas f^{(j)}=\frac{c_1c^{j}e^{2cz}+c_2(-c)^j}{e^{cz}},\eeas where $j=1,2,\ldots,k$. Therefore, \bea\label{e4.17} L(f)=\frac{\sum_{j=1}^{k}b_j(c_jc^je^{2cz}+c_2(-c)^j)}{e^{cz}}.\eea Again from $(\ref{e4.5})$ and $(\ref{e4.13})$, we get \bea\label{e4.18} L(f)=\frac{\sqrt{A}(c_1e^{2cz}-c_2)}{e^{cz}}.\eea Comparing $(\ref{e4.17})$ and $(\ref{e4.18})$, we obtain \bea\label{e4.19} b_1c+b_2c^2+\cdots+b_kc^{k}=\sqrt{A}\eea and
\bea\label{e4.20} -b_1c+b_2c^2-\cdots+(-1)^kb_kc^{k}=-\sqrt{A}.\eea
from $(\ref{e4.19})$ and $(\ref{e4.20})$, it is clear that $A=(b_1c+b_3c^3+\cdots+b_kc^{k})^2$, where $k$ is an odd positive integer.\vspace{1mm}
\par Now from $(\ref{e4.2})$ and $(\ref{e4.13})$, we see that $4d_1d_2=a_1^2(A-1)$ and so \beas 4c_1c_2A=a_1^2(A-1),\eeas where $A=(b_1c+b_3c^3+\cdots+b_kc^{k})^2$, where $k$ is an odd positive integer. This completes the proof of the Theorem \ref{t1}.\end{proof}
\vspace{1mm}
\begin{proof}[\bf Proof of Theorem \ref{t2}]
Let $2\xi$ be the principal branch of $\log h_1$, where $h_1$ is defined as in \ref{e3.23}. Then by Lemma \ref{lem3.8}, we obtain \beas T(r,e^{\xi})=\frac{1}{2}T(r,h_1)+S(r,f)=S(r,f).\eeas Also (\ref{e3.23}) can be written as \bea\label{e4.21} (M(f)-a_1)(M(f)-a_2)=e^{2\xi}(f-a_1)(f-a_2),\eea and so \bea\label{e4.22} G_1H_1=\left(\frac{a_1-a_2}{2}\right)^2(e^{2\xi}-1),\eea where \beas G_1=e^{\xi}f^{\lambda}-\frac{a_1+a_2}{2}e^{\xi}+M(f)-\frac{a_1+a_2}{2}\eeas and \beas H_1=e^{\xi}f-\frac{a_1+a_2}{2}e^{\xi}-M(f)+\frac{a_1+a_2}{2}.\eeas
If $e^{2\xi}\equiv1$, then from (\ref{e4.21}), we get \beas (f^{\lambda}-M(f))(f^{\lambda}+M(f)-a_1-a_2)=0,\eeas which implies that either $f^{\lambda}=M(f)$, or $f^{\lambda}+M(f)=a_1+a_2$.\vspace{1mm}
\par Now suppose that $e^{2\xi}\not\equiv1$. Since $f$ is entire we get $N(r,G_1)+N(r,H_1)=S(r, f)$, and so, from (\ref{e4.22}), we get $N(r,1/H_1)+N(r,1/G_1)=S(r,f)$.
Therefore, \bea\label{e4.23} T\left(r,\frac{G_1^{(j)}}{G_1}\right)+T\left(r,\frac{H_1^{(j)}}{H_1}\right)=S(r,f),\eea where $j=1,2,\ldots,k$.\vspace{1mm}
\par Suppose $(f^{\lambda})^{(1)}=b_1 M(f)$. Then using the condition (\ref{e2.1a}), the Lemma of logarithmic derivative, and the first fundamental theorem of Nevalinna, it is easily seen that $T(r,b_1)=S(r,f)$.
\par From the definition of $G_1$ and $H_1$ it follows that \bea\label{e4.24} G_1+H_1=e^{\xi}(2f^{\lambda}-a_1-a_2)\eea and \bea\label{e4.25} G_1-H_1=2M(f)-a_1-a_2=2\mu (f^{\lambda})^{(1)}-a_1-a_2,\eea where $b\mu=1$ and so $T(r,\mu)=S(r,f)$ as $T(r,b)=S(r,f).$\vspace{1mm}
\par Eliminating $f^{\lambda}$ and $(f^{\lambda})^{(1)}$ from (\ref{e4.24}) and (\ref{e4.25}), we obtain  \bea\label{e4.26} \left(e^{\xi}+\mu\xi^{(1)}-\mu\frac{G_1^{(1)}}{G_1}\right)G_1+\left(\mu\xi^{(1)}-e^{\xi}-\mu\frac{H_1^{(1)}}{H_1}\right)H_1+b_1(a_1+a_2)=0.\eea
Now eliminating $H_1$ from $(\ref{e4.22})$ and $(\ref{e4.26})$, we obtain
\bea\label{e4.27} \chi_{_1} G^{2}+\chi_{_2} G+\chi_{_3}=0,\eea where \bea\label{e4.28} \chi_{_1}=e^{\xi}+\mu\xi^{(1)}-\mu\frac{G_1^{(1)}}{G_1},\eea \bea\label{e4.29}\chi_{_2}= \mu\xi^{(1)}-e^{\xi}-\mu\frac{H_1^{(1)}}{H_1}\left(\frac{a_1-a_2}{2}\right)^2(e^{2\xi}-1),\eea
\bea\label{e4.30} \chi_{_3}=\mu(a_1+a_2).\eea
If $\chi_{_1}\not\equiv0$ or $\chi_{_2}\not\equiv0$, then by Lemma \ref{e3.2}, we get from $(\ref{e4.27})$ that $T(r,G_1)=S(r,f)$, and so from (\ref{e4.22}), we get $T(r,H_1)=S(r,f)$. So, from $(\ref{e4.24})$, we get $T(r,f)=S(r,f)$, which is a contradiction.
Therefore, $\chi_{_1}=\chi_{_2}=0$. Then from $(\ref{e4.27})$, we get $\chi_{_3}=0$. This implies that 
\bea\label{e4.31} e^{\xi}+\mu\xi^{(1)}-\mu\frac{G_1^{(1)}}{G_1}=0,\eea	
\bea\label{e4.32} \mu\xi^{(1)}-e^{\xi}-\mu\frac{H_1^{(1)}}{H_1}=0,\eea	
\bea\label{e4.33} a_1+a_2=0.\eea	
Adding $(\ref{e4.31})$ and $(\ref{e4.32})$, we get \beas \frac{G_1^{(1)}}{G_1}+\frac{H_1^{(1)}}{H_1}=2\xi^{(1)}, \eeas and so by integration , we have \bea\label{e4.34}G_1H_1=c_0^{*}e^{2\xi},\eea where $c_0^{*}$ is a non-zero constant.\vspace{1mm} 
\par Now from $(\ref{e4.22})$, $(\ref{e4.33})$ and $(\ref{e4.34})$, we get $e^{2\xi}=A$, where $A$ is a constant.\vspace{1mm}
\par From $(\ref{e4.24})$, $(\ref{e4.25})$ and $(\ref{e4.33})$, we get \bea\label{e4.35} \left(\frac{\sqrt{A}}{\mu}-\frac{G_1^{(1)}}{G_1}\right)G_1^{2}=-\left(\frac{\sqrt{A}}{\mu}-\frac{H_1^{(1)}}{H_1}\right)B,\eea where $B=(a_1-a_2)^2/4(A-1)$, constant.\vspace{1mm}
\par If $\sqrt{A}/\mu-G_1^{(1)}/G_1\not\equiv0,$ then from $(\ref{e4.23})$ and $(\ref{e4.35})$, we get $T(r,G_1)=S(r,f)$ and so from $(\ref{e4.34})$, we get $T(r,H_1)=S(r,f)$. Therefore, from $(\ref{e4.24})$, we get $T(r,f)=S(r,f)$, which is a contradiction. Hence we must have $\mu G_1^{(1)}-\sqrt{A}G_1=0$ and $\mu H_1^{(1)}-\sqrt{A}H_1=0$. This implies by Lemma \ref{lem3.3} that $G_1$ and $H_1$ are of finite order. Also from $(\ref{e4.34})$, we see that $G_1$ and $H_1$ do not assume the value $0$.\vspace{1mm}
\par Therefore, we may assume that $G_1=e^{P}$ and $H_1=e^{Q}$, where $P,\; Q$ are polynomials of degree $p$ and $q$, respectively.\vspace{1mm}
\par Differentiating once, we get $G_1^{(1)}=P^{(1)}e^{P}$ and $H_1^{(1)}=Q^{(1)}e^{Q}$. Therefore, $P^{(1)}$ and $Q^{(1)}$ are polynomials of degree $(p-1)$ and $(q-1)$, respectively. Since $\mu G_1^{(1)}=\sqrt{A}G_1$ and $\mu H_1^{(1)}=\sqrt{A}H_1$, we have $p=q=1$. Hence in view of $(\ref{e4.34})$, we may write $G_1=2d_1^{*}e^{cz}$ and $H=2d_2^{*}e^{-cz}$, where $c,\;d_1^{*},\;d_2^{*}$ are non-zero constants.\vspace{1mm}
\par Now from $(\ref{e4.24})$, $(\ref{e4.25})$ and  $(\ref{e4.33})$, we get \beas f^{\lambda}=c_1e^{cz}+c_2e^{-cz},\; M(f)=\frac{\sqrt{A}(c_1e^{2cz}-c_2)}{e^{cz}},\eeas where $c_1=d_1^{*}/\sqrt{A}$ and $c_2=d_2^{*}/\sqrt{A}$. This completes the proof of the theorem.\vspace{1mm}

\end{proof}





\end{document}